\documentclass[a4paper,11pt,reqno]{amsart}
\usepackage{amsmath}
\usepackage{amssymb}

\theoremstyle{definition}
\newtheorem{theorem}{Theorem}[section]
\newtheorem{proposition}[theorem]{Proposition}
\newtheorem{lemma}[theorem]{Lemma}

\title[Incipient cluster in Ising percolation]{Incipient infinite cluster \\in 2D Ising percolation}
\author[Y. Higuchi]{Yasunari Higuchi}\thanks{Supported by JSPS Grant-in-Aid for Scientific Research (C) No. 23540136. (Y. H.)}
\address{Department of Mathematics, Kobe University, 1-1 Rokko, Kobe 657-8501, Japan.}
\email{higuchi@math.kobe-u.ac.jp}        
\author[K. Kinoshita]{Kazunari Kinoshita}
\address{Department of Mathematics, Kobe University, 1-1 Rokko, Kobe 657-8501, Japan.}
\author[M. Takei]{Masato Takei}\thanks{Partially supported by Osaka Electro-Communication University. (M. T.)}
\address{Department of Applied Mathematics, Faculty of Engineering, Yokohama National University, Hodogaya, Yokohama 240-8501, Japan.}
\email{takei@ynu.ac.jp}
\author[Y. Zhang]{Yu Zhang}
\address{Department of Mathematics, University of Colorado, Colorado Springs, CO 80933, USA.}
\email{yzhang3@uccs.edu}
\subjclass[2000]{Primary~60K35, Secondary~82B43}
\keywords{Percolation, Ising model, Incipient infinite cluster.}

\numberwithin{equation}{section}

\usepackage{mathrsfs}
\renewcommand{\mathcal}{\mathscr}

\begin{document}
\maketitle
\begin{abstract}
We consider the percolation problem in the high-temperature Ising model on the two-dimensional square lattice at or near critical external fields. The incipient infinite cluster (IIC) measure in the sense of Kesten is constructed. As a consequence, we can obtain some geometric properties of IIC. The result holds also for the triangular lattice.
\end{abstract}

\section{Introduction}
We consider the square lattice $\mathbf{Z}^2$ and the sample space $\Omega := \{-1,+1\}^{\mathbf{Z}^2}$
of spin configurations on $\mathbf{Z}^2$. 
The spin value at $x \in \mathbf{Z}^2$ in
the configuration $\omega \in \Omega$ is denoted by $\omega (x)$.
Let $|x|$ denote the $\ell^1$-norm of $x \in \mathbf{Z}^2$:
\[ |x| := |x^1| + |x^2| \mbox{ for } x= (x^1,x^2) \in \mathbf{Z}^2. \]
For any finite $V$, we define the Hamiltonian $H_{V,h}^{\omega}: \Omega_V =  \{ -1,+1\}^{V} \to \mathbf{R}$ by
\[ H_{V,h}^{\omega} (\sigma) = -\frac{1}{2} \sum_{x,y \in V,\,|x-y|=1} \sigma(x)\sigma(y) - \sum_{x \in V} \left(h + \sum_{y \notin V,\,|x-y|=1} \omega(y) \right) \sigma (x), \]
for $\sigma \in \Omega_V$. Here $h$ is a real number called the {\it external field}.
We then define the finite Gibbs measure on $\Omega_V$ by
\[ q_{V,\beta,h}^{\omega} (\sigma) = \left[ \sum_{\sigma' \in \Omega_V} \exp \{ -\beta H_{V,h}^{\omega} (\sigma') \} \right]^{-1} \exp \{ -\beta H_{V,h}^{\omega} (\sigma) \}. \]
Here $\beta$ is a positive number called the {\it inverse temperature}. For any set $V \subset \mathbf{Z}^2$, we denote by $\mathcal{F}_V$ the $\sigma$-algebra generated by $\{ \omega(x) : x \in V \}$.
For each $\beta > 0$ and $h \in \mathbf{R}$, the Gibbs measure is a probability measure $\mu_{\beta,h} $ on $\Omega$ in the sense of the following DLR equation:
\[ \mu_{\beta,h} (\;\cdot\,|\,\mathcal{F}_{V^c})(\omega) = q_{V,\beta,h}^{\omega} (\;\cdot\;) \quad \mbox{$\mu_{\beta,h}$-almost every $\omega$}, \]
where $V^c = \mathbf{Z}^2 \setminus V$. Let $\beta_c$ be the critical value such that if $\beta < \beta_c$ or $h \neq 0$, the Gibbs measure is unique for $(\beta, h)$.

In this paper, we consider the percolation problem in the Ising
model in the high-temperature regime, and construct the incipient infinite cluster (IIC) measure in the sense of Kesten \cite{K86i}.
A sequence $\{x_1, x_2, \ldots , x_n\}$ of points
in $\mathbf{Z}^2$ is called a {\it path} if $|x_i - x_{i+1}| = 1$ for $ 1 \leq i \leq n - 1$.
A path is called a $(+)$-{\it path} if the spin value is $+$ for every point of this path.
A $(+)$-{\it cluster} is the set of vertices connected by $(+)$-paths.
Let ${\bf C}_0^+$ be the $(+)$-cluster that contains the origin, and
$\# {\bf C}_0^+$ denotes the number of vertices in ${\bf C}_0^+$.
Let $S(n)$ be the square $[-n,n]^2$, and $S^c(n):=(S(n))^c$.

For $\beta >0$, we define $h_c(\beta)$ by
\[ h_c(\beta) := \inf \{ h: \mu_{\beta, h}( \#{\bf C}_0^+=\infty)>0\}. \]
It has been proved (see \cite{Hig93a}) that if $\beta < \beta_c$,
then $h_c(\beta)> 0$. 
We write $v \rightsquigarrow w$ if there exists a $(+)$-path from $v$ to $w$.
Similarly $v \rightsquigarrow B$ (resp. $A\rightsquigarrow B$) means that
$v \rightsquigarrow w$ for some $w \in B$ (resp. for some $ v \in A$ and $w \in B$).

Our main result is the following:

\begin{theorem} \label{thm:1.1}
Let $\beta < \beta_c$. For every cylinder event $E$, the limits
\begin{equation} \label{eq:thm1.1(1)}
\lim_{n \to \infty} \mu_{\beta,h_c(\beta)} \bigl(E\, \bigl|\, \mathbf{O} \rightsquigarrow S^c(n)  \bigr)
\end{equation}
and
\begin{equation} \label{eq:thm1.1(2)}
\lim_{h \searrow h_c(\beta)} \mu_{\beta,h} \bigl(E\, \left|\, \# {\bf C}_0^+=\infty \right. \bigr)
\end{equation}
exist and are equal. If we denote their common value by $\nu(E)$, then $\nu$ extends
uniquely to a probability measure on $\Omega$, and
\begin{equation} \label{eq:thm1.1(3)}
\nu \left(\begin{array}{@{\,}c@{\,}}
\mbox{there exists exactly one infinite $(+)$-cluster $\tilde{\bf C}_0^+$,} \\
\mbox{and $\tilde{\bf C}_0^+$ contains the origin $\mathbf{O}$} \\
\end{array} \right)
= 1.
\end{equation}
\end{theorem}

Kesten's proof \cite{K86i} uses independence in many places; we have to replace it with the mixing property (Theorem \ref{thm:2.2} below). Still the argument is relatively simple.

The expectation with respect to $\nu$ is denoted by $E_{\nu}$. The following theorem is obtained by a similar method as \cite{HTZ12}, Lemma 2.3 (see also \cite{HTZ10}). 

\begin{theorem}[cf. \cite{K86i} (8)] For any $t \geqq 1$, 
\[ E_{\nu} \left[ \{ \# (\tilde{\bf C}_0^+ \cap S(n)) \}^t \right] \asymp \{ n^2 \mu_{\beta,h_c(\beta)} (\mathbf{O} \rightsquigarrow S^c(n)) \}^t, \]
where $f(n) \asymp g(n)$ means that $C_1 g(n) \leq f(n) \leq C_2 g(n)$.
\end{theorem}

Our proof works also for Ising percolation on the triangular lattice with $\beta<\beta_c$ and $h=h_c(\beta)=0$. (See \cite{BCM10} for a related work.) In view of universality, we believe that $\mu_{\beta, h_c(\beta)}(\mathbf{O} \rightsquigarrow S^c(n)) \approx n^{-5/48}$ as in the critical percolation on the triangular lattice. 

\section{Notation and preliminary results}

A path $\{x_1, x_2, \ldots , x_n\}$ is called {\it self-avoiding} if $x_i \neq  x_j$ for $i \neq  j$.
It is called a {\it circuit} if $\{x_1, x_2, \ldots , x_{n-1}\}$ is self-avoiding and $x_n = x_1$.
For a given circuit $\mathcal{C}$, and a finite set $D\subset \mathbb{Z}^2$,
we say that $\mathcal{C}$ surrounds $D$ if $\mathcal{C}\subset D^c$ and
any path connecting $D$ with $\infty$ intersects $\mathcal{C}$.

For $V_1,\,V_2 \subset \mathbf{Z}^2$, $d(V_1,V_2)$ denotes the $\ell^1$-distance between $V_1$ and $V_2$; that is,
\[ d(V_1,V_2) = \inf\{ |x-y| : x \in V_1,\,y \in V_2 \}. \]

The following is a refinement of Theorem 2 (ii) of \cite{Hig93a}, which can be obtained without changing the original proof. Although it is stated for the square lattice, it is also valid for the triangular lattice, together with the exponential decay result in \cite{ABF87}.

\begin{theorem}[The mixing property] \label{thm:2.2}
Let $\beta <\beta_c$, and $h\geq h_c(\beta )$.
There exist constants $C>0$ and $\alpha >0$ such that the following holds.
Let $l$ be a given integer. Assume that
$V\subset \Lambda $ are finite subsets of ${\mathbf Z}^2$,
$A \in {\mathcal F}_V$ and $\omega_1, \omega_2 \in \Omega $
satisfy 
\[ \omega_1(x) = \omega _2(x)=+1 \]
for every $x\in \partial\Lambda $ with $d(x,V)<\ell $.
Then we have
\begin{equation}\label{eq:fine mixing}
| q_{\Lambda ,\beta ,h}^{\omega_1}(A)- 
     q_{\Lambda ,\beta ,h}^{\omega_2}(A)|
     \leq C \sum_{x\in V}
         \sum_{y\in \partial \Lambda, d(y,V)\geq \ell }
           e^{-\alpha |x-y|}.
\end{equation}
\end{theorem}
In particular, 
for every pair of finite subsets $V$ and $W$ of $\mathbf{Z}^2$ with $V \subset W$, 
\begin{align} \label{HZ(1.1)}
&\sup_{\omega \in \Omega,\,A \in \mathcal{F}_V} \left| \mu_{\beta,h}(A) - \mu_{\beta,h}(A\,|\,\mathcal{F}_{W^c} )(\omega)\right| \\
&\leq C|V| d(V,W^c)\exp\{-\alpha d(V,W^c)\}. \notag
\end{align}
Hereafter we fix $\beta < \beta_c$.

\begin{lemma}[\cite{Hig93b}] \label{lemma:3.3}
For any integer $k > 0$, there exists a constant $\delta_k>0$ such that for all $h \geq h_c(\beta)$ and for all $n$,
\begin{align}
&\mu_{\beta,h} \bigl( \mbox{there exists a horizontal $(+)$-crossing of $[0,kn] \times [0,n]$} \bigr)\geq \delta_k, \label{eq:(3.17)} \\
&\mu_{\beta,h} \bigl( \mbox{there exists a vertical $(+)$-crossing of $[0,n] \times [0,kn]$} \bigr) \geq \delta_k. \label{eq:(3.18)}
\end{align}
\end{lemma}

\begin{lemma} \label{lem:one-arm} 
There exists a positive constant $C_1$ such that for $R<n$,
\[ \mu_{\beta,h_c(\beta)}  \bigl( S(R) \rightsquigarrow S^c(n) \bigr) \geq C_1\dfrac{R}{n}. \]
\end{lemma}

For the idea of the proof of Lemma \ref{lem:one-arm}, see e.g. \cite{Z95}, Lemma 5.

\section{Proof of Theorem \ref{thm:1.1}}
The proof of Theorem \ref{thm:1.1} goes parallel to 
Kesten's original proof \cite{K86i}, except one point 
in the proof of Lemma \ref{lem:3.2} below,
where the independence played an important role.
This can be overcome by means of the fine part of
the mixing property (\ref{eq:fine mixing}).

\begin{lemma}[cf. \cite{K86i} p.374] We can find $1 \leq k(1) < k(2) < \cdots$ such that
\[ \alpha_i := \mu_{\beta,h_c(\beta)}\left( \begin{array}{@{\,}c@{\,}}
\mbox{there exists a $(+)$-circuit surrounding $ S(3^{k(i)})$}\\
\mbox{in the annulus $S(3^{k({i+1})}) \setminus  S(2\cdot 3^{k(i)})$}
\end{array}
\right) \]
tends to $1$ as $i \to \infty$.
\end{lemma}
This lemma can be obtained by Lemma \ref{lemma:3.3}, the FKG inequality and the mixing property (for a similar argument, see \cite{Hig93a}, Lemma 5.3).
We fix a sequence $\{ k(i) \}$ stated above, and put
\[ A(i) := S(3^{k(i+1)}) \setminus S(2 \cdot 3^{k(i)}). \]
The set of circuits surrounding the origin is denoted by $\Sigma$. For $i=1,2,\ldots$
\[ \Sigma (i) := \{ \mathcal{C} \in \Sigma : \mathcal{C} \subset A(i) \} \]
is the set of circuits in $A(i)$ surrounding $S(2 \cdot 3^{k(i)})$. For a circuit $\mathcal{C} \in \Sigma (i)$, we define
\[ F_i (\mathcal{C}) := \{ \mbox{$\mathcal{C}$ is the innermost $(+)$-circuit in $A(i)$ surrounding $S(2 \cdot 3^{k(i)})$} \}, \]
and
\[ F_i := \{ \mbox{there exists a $(+)$-circuit in $\Sigma (i)$} \}=\bigcup_{\mathcal{C} \in \Sigma (i)} F_i (\mathcal{C}).  \]
Then, 
\begin{equation} \label{Kinc(15)}
\displaystyle \alpha_i =\mu_{\beta,h_c(\beta)} \bigl( F_i \bigr) = \sum_{\mathcal{C} \in \Sigma (i)} \mu_{\beta,h_c(\beta)}  \bigl( F_i (\mathcal{C}) \bigr).
\end{equation}

Now let $E$ be any cylinder set depending only on spins in $S(l)$, and let $l <3^{k(i)} < 3^{k(i+1)} < n$. Then
\begin{align*}
&E \cap \{ \mathbf{O} \rightsquigarrow S^c(n) \} \\
&=\bigl( E \cap F_i^c \cap \{ \mathbf{O}  \rightsquigarrow S^c(n) \}\bigr) \cup \left(\bigcup_{\mathcal{C} \in \Sigma(i)} E \cap F_i(\mathcal{C}) \cap \{ \mathbf{O}  \rightsquigarrow S^c(n) \} \right).
\end{align*}
By the Markov property, we have
\begin{align}
&\mu_{\beta,h} \bigl( E \cap F_i(\mathcal{C}) \cap \{ \mathbf{O}  \rightsquigarrow S^c(n) \} \bigr) \notag \\
&= \mu_{\beta,h} \bigl(E \cap F_i(\mathcal{C}) \cap \{ \mathbf{O}  \rightsquigarrow \mathcal{C}  \} \bigr) 
\mu_{\beta,h} \bigl(\mathcal{C} \rightsquigarrow S^c(n) \,|\, [ \mathcal{C}]_+ \bigr), \label{eq:(3.2)}
\end{align}
where we put
\[
[\mathcal{C}]_+ = 
\{ \omega \in \Omega \, ; \, \omega (x) =+1, x \in \mathcal{C} \} . 
\]
Let
\[ \gamma(\mathcal{C},n) :=  \mu_{\beta,h} \bigl(\mathcal{C} \rightsquigarrow S^c(n)  \,|\, [\mathcal{C}]_+ \bigr).\]
Then by the FKG inequality, for $h \geq h_c(\beta)$,
\begin{align}
&\left| \mu_{\beta,h} \bigl( E \cap \{ \mathbf{O} \rightsquigarrow S^c(n)\} \bigr) 
- \sum_{\mathcal{C} \in \Sigma (i)} \mu_{\beta,h} \bigl(E \cap F_i(\mathcal{C}) \cap \{ \mathbf{O} \rightsquigarrow \mathcal{C} \} \bigr) \gamma(\mathcal{C},n) \right| \notag \\
&\leq \mu_{\beta,h} \bigl( F_i^c \cap \{ \mathbf{O} \rightsquigarrow S^c(n) \} \bigr) \leq (1-\alpha_i) \mu_{\beta,h} \bigl( \mathbf{O} \rightsquigarrow S^c(n) \bigr) . \label{eq:(3.3)}
\end{align}

In the same way, we obtain the following: For
$i<j$, $\mathcal{C}\in \Sigma (i)$ and $\mathcal{D}\in \Sigma(j)$, let
\[ M(\mathcal{C},\mathcal{D},j) := \mu_{\beta,h} \bigl( F_j(\mathcal{D}) \cap \{ \mathcal{C} \rightsquigarrow \mathcal{D} \} \,|\, [\mathcal{C}]_+  \bigr). \]
Then for $\mathcal{C} \in \Sigma (i)$ and $3^{k({i+1})} < 3^{k(j)} < 3^{k({j+1})} < n$, we have
\begin{align}
&\left| \gamma(\mathcal{C},n) - \sum_{\mathcal{D} \in \Sigma (j)} M(\mathcal{C},\mathcal{D},j) \gamma(\mathcal{D},n) \right| \notag \\
&\leq \mu_{\beta,h} ( F_j^c \,|\, [\mathcal{C}]_+ ) \gamma(\mathcal{C},n) \leq (1-\alpha_j) \gamma(\mathcal{C},n). \label{eq:(3.5)}
\end{align}
The last inequality is by the FKG inequality.
\begin{lemma}[cf. \cite{K86i} p.376]
If for any $\mathcal{C}',\mathcal{C}'' \in \Sigma(i)$
\begin{equation} \label{eq:(3.6)}
\lim_{n \to \infty} \dfrac{\gamma (\mathcal{C}',n)}{\gamma (\mathcal{C}'',n)} 
\end{equation}
exists, then the limit \eqref{eq:thm1.1(2)} exists.
\end{lemma}
By \eqref{eq:(3.5)}, for fixed $i$ and $\varepsilon > 0$, we can find a $j$ such that
\[ e^{-\varepsilon} \gamma (\mathcal{C},n) \leq \sum_{\mathcal{D} \in \Sigma (j)} M(\mathcal{C},\mathcal{D},j) \gamma(\mathcal{D},n)  \leq e^{\varepsilon} \gamma (\mathcal{C},n) , \]
uniformly in $\mathcal{C} \in \Sigma(i)$ and $h \geq h_c(\beta)$. 
Iterating this argument for $\varepsilon 2^{-(s-1)}$ at the $s$-th step,
we can find $i \leq j_1 < j_2 < \cdots < j_s$, depending only on $i$ and $\varepsilon$, such that
\begin{align}
&e^{-2\varepsilon} \gamma (\mathcal{C},n) \notag \\
&\leq \sum_{\mathcal{D}_1 \in \Sigma ({j_1})}\sum_{\mathcal{D}_2 \in \Sigma ({j_2})} \cdots\sum_{\mathcal{D}_s \in \Sigma ({j_s})} \left( \prod_{m=1}^{s} M(\mathcal{D}_{m-1},\mathcal{D}_m,j_m)
\right)  \gamma(\mathcal{D}_s,n) \notag \\
&\leq e^{2\varepsilon} \gamma (\mathcal{C},n) \label{eq:(3.15)}
\end{align}
for all $h \geq h_c(\beta)$ and $n > 3^{k(j_s+1)}$. Here we put $\mathcal{D}_0=\mathcal{C}$.

We further assume that these subsequences $\{ j_s\}$ are chosen sufficiently
large so that for each $m\geq 1$, we can find some integer $t$
such that $3^{k(j_m+1)}<3^{-2}t<t<3^3t<3^{k(j_{m+1})}$, and
\begin{equation}\label{eq:cond for js}
\frac{C}{C_1}(2t)^4 e^{-2\alpha t}< \frac{1}{2}\min_{{\mathcal D}\in \Sigma (j_m)}
\mu_{\beta ,h}([\mathcal{D}]_+),
\end{equation}
uniformly in $h\geq h_c(\beta )$.
Indeed, since $\# \mathcal{D} \leq w(m):=3^{2k(j_m+1)}$ for any $\mathcal{D}\in \Sigma (j_m)$, we have
\[
\min_{{\mathcal D}\in \Sigma (j_m)}\mu_{\beta ,h}([\mathcal{D}]_+)\geq 
\mu_{\beta ,h}(\omega (\mathbf{O})=+1)^{w(m)}\geq \left(\frac{1}{2}\right)^{w(m)}
\]
by the FKG inequality. This shows that we can obtain \eqref{eq:cond for js} when we choose a suitable $t$ of the same order as $w(m)$.

The following proposition is essentially the key to the proof of Theorem
\ref{thm:1.1}, and after this proposition, the argument is the same
as in \cite{K86i}. 
\begin{proposition} \label{lem:3.2}
There exists a constant $1 < \kappa < \infty$ (independent of $\varepsilon $ and $\{j_m\}$ chosen above) such that for all $h \geq h_c(\beta)$, $\mathcal{D}', \mathcal{D}'' \in \Sigma(j_m)$,
\[ \dfrac{M(\mathcal{D}',\mathcal{E}',j_m)M(\mathcal{D}'',\mathcal{E}'',j_m)}{M(\mathcal{D}',\mathcal{E}'',j_m)M(\mathcal{D}'',\mathcal{E}',j_m)} \leq \kappa^2 \]
for every $m\geq 1$.
\end{proposition}
\begin{proof}
Let $\partial S(n)=S(n+1) \setminus S(n)$.
We shall prove that there exists a $\kappa > 1$ such that for 
$\mathcal{D} \in \Sigma (j_{m-1}),\, \mathcal{E} \in \Sigma (j_m)$ one has 
\begin{align}
& \kappa^{-1} \gamma(\mathcal{D},t) \mu_{\beta,h} \left( F_{j_m}(\mathcal{E}),\,
\partial S(3t) \rightsquigarrow \mathcal{E} \right) \notag\\
&\leq M(\mathcal{D},\mathcal{E},j_m) \notag \\
&\leq \kappa \gamma(\mathcal{D},t) \mu_{\beta,h} \left( F_{j_m}(\mathcal{E}),\,
\partial S(3t) \rightsquigarrow \mathcal{E} \right). \label{eq:(3.21)}
\end{align}
Here, $t$ is the integer satisfying
$3^{k(j_m+1)}<3^{-2}t<t<3^3t<3^{k(j_{m+1})}$, and 
the condition (\ref{eq:cond for js}).

First we prove the second inequality in \eqref{eq:(3.21)}. 

Since 
$
\{ \mathcal{D} \rightsquigarrow \mathcal{E} \} \subset
\{ \mathcal{D} \rightsquigarrow S^c(t)\} \cap \{ \partial S(3t) 
\rightsquigarrow \mathcal{E}\}
$, we have
\begin{align*}
&\mu_{\beta ,h}( 
[\mathcal{D}]_+\cap \{ \mathcal{D} \rightsquigarrow \mathcal{E} \}\cap 
F_{j_m}(\mathcal{E}))\\
&\leq \mu_{\beta ,h}\left( 
[\mathcal{D}]_+\cap
\{ \mathcal{D} \rightsquigarrow S^c(t)\} \mid \{ \partial S(3t) 
\rightsquigarrow \mathcal{E}\} \cap  F_{j_m}(\mathcal{E}) \right) \\
&\times \mu_{\beta ,h}\left(\{ \partial S(3t) 
\rightsquigarrow \mathcal{E}\} \cap  F_{j_m}(\mathcal{E}) \right).
\end{align*}
By the mixing property (\ref{HZ(1.1)}),
\begin{align*}
&\mu_{\beta ,h}\left( 
[\mathcal{D}]_+\cap
\{ \mathcal{D} \rightsquigarrow S^c(t)\} \mid \{ \partial S(3t) 
\rightsquigarrow \mathcal{E}\} \cap  F_{j_m}(\mathcal{E}) \right) \\
&\leq \mu_{\beta ,h}( [\mathcal{D}]_+\cap 
\{ \mathcal{D} \rightsquigarrow S^c(t)\} 
       ) +C(2t)^3e^{-\alpha \cdot 2t}.
\end{align*}
But by the FKG inequality and by Lemma \ref{lem:one-arm}, 
the first term in the right hand side is not less than
\[
\mu_{\beta ,h}([\mathcal{D}]_+)C_1t^{-1},
\]
which is not less than $ 2C(2t)^3e^{-\alpha \cdot 2t}$ by (\ref{eq:cond for js}).
Hence, we see that
\begin{align*}
&\mu_{\beta ,h}\left( 
[\mathcal{D}]_+\cap
\{ \mathcal{D} \rightsquigarrow S^c(t)\} \mid \{ \partial S(3t) 
\rightsquigarrow \mathcal{E}\} \cap F_{j_m}(\mathcal{E}) \right) \\
&\leq \frac{3}{2} \mu_{\beta ,h}([\mathcal{D}]_+\cap 
\{ \mathcal{D} \rightsquigarrow S^c(t)\} ).
\end{align*}
Thus, we have
\begin{align*}
&\mu_{\beta ,h}( 
[\mathcal{D}]_+\cap \{ \mathcal{D} \rightsquigarrow \mathcal{E} \}\cap 
F_{j_m}(\mathcal{E}))\\
&\leq \frac{3}{2}\mu_{\beta ,h}\left( 
[\mathcal{D}]_+ \cap \{ \mathcal{D} \rightsquigarrow S^c(t)\} \right) 
\mu_{\beta,h}\left( F_{j_m}(\mathcal{E})\cap \{ \partial S(3t) 
\rightsquigarrow \mathcal{E}\} \right) .
\end{align*}
Dividing both sides by $\mu_{\beta ,h}([\mathcal{D}]_+)$,
we obtain the desired inequality for $\kappa \geq \frac{3}{2}$.

For the first inequality in  (\ref{eq:(3.21)}), let $\mathcal{E}\in \Sigma (j_m)$, and let $\Theta $ be the region in
$A(j_m)$ inside $\mathcal{E}$, namely,
\[
\Theta =\left\{ x \in A(j_m) \, ; \, \begin{array}{@{\,}c@{\,}}
       \mbox{any path connecting $x$ with} \\
        S^c(3^{k(j_m+1)}) \mbox{ intersects }\mathcal{E}
    \end{array}\right\} .
\]
Let $\Xi $ be the configuration in $\Theta $, and we write $[\Xi ]$
for the cylinder set defined by $\Xi$,
\[
[\Xi ] := \{ \omega \in \Omega \, ; \, \omega (x) = \Xi (x), \, x\in \Theta \} .
\]
Fix a configuration $\Xi$ such that $[\Xi ] \subset F_{j_m}(\mathcal{E})$. Set
$u= 3^{-1}t, w=9t<3^{k(j_m+1)}$, and $G=G_1\cap G_2\cap G_3$, where
\begin{align*}
G_1&=\{ \mbox{there exists a $(+)$-circuit surrounding $S(u)$ in $S(t)$}\} , \\
G_2&=\{\mbox{there exists a $(+)$-circuit surrounding $S(3t)$ in $S(w)$}\} ,\\ 
G_3&=\{ \partial S(u)  \rightsquigarrow S^c(w) \} .
\end{align*}
Also, let
\[
H_1=\{  \mathcal{D}\rightsquigarrow S^c(t) \} , \mbox{ and } \,
H_2=\{ \partial S(3t) \rightsquigarrow \mathcal{E}\}.
\]
Then since
\[
\{ \mathcal{D}\rightsquigarrow \mathcal{E} \} \supset G\cap H_1\cap H_2,
\]
we have by the FKG inequality 
\begin{align}
\mu_{\beta ,h}( 
[\mathcal{D}]_+\cap \{ \mathcal{D} \rightsquigarrow \mathcal{E} \}\cap [\Xi ])
&\geq \mu_{\beta ,h}( [\mathcal{D}]_+\cap  G\cap H_1\cap H_2\cap [\Xi] )
\notag\\
&\geq \mu_{\beta ,h}( G\mid [\mathcal{D}]_+\cap [\Xi])
\notag\\
&\times \mu_{\beta ,h}( H_1\mid  [\mathcal{D}]_+\cap [\Xi])\notag\\
&\times \mu_{\beta ,h}(H_2\mid  [\mathcal{D}]_+\cap [\Xi])\notag\\
&\times \mu_{\beta ,h}( [\mathcal{D}]_+\cap [\Xi]) .\label{eq:decomposition}
\end{align}
Further by the FKG inequality, 
\begin{align}
\mu_{\beta ,h}( G \mid [\mathcal{D}]_+ \cap [\Xi])&\geq \mu_{\beta ,h}(G\mid [\Xi]), \label{eq:circular-connection}\\
\mu_{\beta ,h}( H_2\mid [\mathcal{D}]_+ \cap [\Xi])&\geq \mu_{\beta ,h}(H_2\mid [\Xi]). \label{eq:outside 3t}
\end{align}
By the mixing property (\ref{HZ(1.1)}), the
right hand side of (\ref{eq:circular-connection}) is not less than
\[
\mu_{\beta ,h}(G)- C\frac{2u}{3}(9t)^2 e^{\alpha \frac{2u}{3}} \geq \frac{\delta_3^8\delta_{14}}{2}=:C_3>0.
\]
As for $H_1$, since the spins on $\mathcal{D}$ are all positive in 
$[\mathcal{D}]_+$, the mixing property (\ref{eq:fine mixing}) ensures that
\[
\mu_{\beta ,h}( H_1\mid [\mathcal{D}]_+\cap [\Xi] ) 
\geq \mu_{\beta ,h}( H_1\mid [\mathcal{D}]_+) - C(2t)^3 e^{-\alpha \cdot 2t}.
\]
By the FKG inequality and Lemma \ref{lem:one-arm},
\[
\mu_{\beta ,h}( H_1\mid [\mathcal{D}]_+)\geq \mu_{\beta ,h}(H_1)\geq C_1t^{-1}.
\]
Thus, we have by  (\ref{eq:cond for js}), 
\begin{equation}\label{eq:estimate for H_1}
\mu_{\beta ,h}( H_1\mid [\mathcal{D}]_+\cap [\Xi] ) \geq
\frac{1}{2}\mu_{\beta ,h}( H_1\mid [\mathcal{D}]_+ ).
\end{equation}

On the other hand, again by (\ref{eq:cond for js}),
\[
\mu_{\beta ,h}( [\mathcal{D}]_+\mid [\Xi]) \geq \mu_{\beta ,h}([\mathcal{D}]_+)
-C(2t)^3e^{-\alpha \cdot 2t}\geq \frac{1}{2}\mu_{\beta ,h}([\mathcal{D}]_+).
\]
From this we have
\begin{equation}\label{eq:estimate for [Xi]cap [D]}
\mu_{\beta ,h}([\mathcal{D}]_+\cap [\Xi])\geq \frac{1}{2}
\mu_{\beta ,h}([\mathcal{D}]_+)\mu_{\beta ,h}([\Xi])
\end{equation}
Combining (\ref{eq:decomposition})--(\ref{eq:estimate for [Xi]cap [D]}),
we have
\begin{align*}
&\mu_{\beta, h}(([\mathcal{D}]_+\cap \{ \mathcal{D}\rightsquigarrow\mathcal{E}\}
\cap [\Xi] )
\geq \frac{C_3}{4}\mu_{\beta ,h}(H_1\cap [\mathcal{D}]_+)\mu_{\beta ,h}(H_2\cap
[\Xi]).
\end{align*}
Summing this up with respect to $\Xi$ such that $[\Xi]\subset F_{j_m}(\mathcal{E})$,
and then dividing both sides by $\mu_{\beta ,h}([\mathcal{D}]_+)$,
we obtain the desired inequality for $\kappa \geq 4C_3^{-1}$.
\end{proof}

\end{document}